\documentclass[12pt]{article}

\usepackage[T1]{fontenc}
\usepackage[utf8]{inputenc}
\usepackage{lmodern}

\usepackage[left=1.05in,right=1.05in,top=1in,bottom=1in]{geometry}
\usepackage{amsmath,amssymb,amsfonts,amsthm,mathtools}
\usepackage{bm}
\usepackage{graphicx}
\usepackage{float}
\usepackage{tikz}
\usetikzlibrary{arrows.meta,calc,positioning,decorations.markings}
\usepackage{enumitem}
\usepackage{hyperref}
\usepackage[noabbrev]{cleveref}
\usepackage{microtype}

\hypersetup{
  colorlinks=true,
  linkcolor=blue!50!black,
  citecolor=blue!60!black,
  urlcolor=blue!60!black
}

\numberwithin{equation}{section}

\newcommand{\R}{\mathbb{R}}
\newcommand{\C}{\mathbb{C}}
\newcommand{\HH}{\mathbb{H}^3(-1)}
\newcommand{\tr}{\mathrm{tr}}
\newcommand{\SL}{\mathrm{SL}(2,\C)}
\newcommand{\SU}{\mathrm{SU}(2)}
\newcommand{\Herm}{\mathrm{Herm}(2)}

\newtheorem{theorem}{Theorem}[section]
\newtheorem{proposition}[theorem]{Proposition}
\newtheorem{lemma}[theorem]{Lemma}
\newtheorem{corollary}[theorem]{Corollary}
\theoremstyle{definition}
\newtheorem{definition}[theorem]{Definition}
\newtheorem{remark}[theorem]{Remark}
\newtheorem{example}[theorem]{Example}

\title{A Weierstrass-Kenmotsu Type Representation for CMC
\texorpdfstring{$0\le H<1$}{0<=H<1} in
\texorpdfstring{$\mathbb{H}^3(-1)$}{H3(-1)}}
\author{Magdalena Toda, Erhan Güler, Madusha Dilhani Atampalage}
\date{}

\begin{document}

\maketitle

\begin{abstract}
We develop a Weierstrass-Kenmotsu type representation for conformal immersions of constant mean curvature $0\le H<1$ in hyperbolic $3$-space $\HH$. The construction is based on the Hermitian model of $\HH$, a balanced spectral deformation, and Iwasawa splitting of $\SL$.

We show that such immersions arise locally from a rank-one $(1,0)$-form $\eta$ and a constant complex parameter $\lambda\in\C^*$ through a flat $\SL$-connection of the form
\[
S^{-1}dS=\eta-\lambda\,\eta^*,
\]
with mean curvature
\[
H=\frac{1-|\lambda|^2}{1+|\lambda|^2}.
\]
Conversely, every conformal CMC immersion with $0\le H<1$ is locally obtained from such flat rank-one data.

We establish an explicit correspondence with the representation of Aiyama and Akutagawa via a gauge transformation, and interpret the construction in terms of Kokubu's adjusted normal Gauss map. We further discuss the role of the flatness condition, present simple local and cylindrical model examples, and outline aspects of monodromy and numerical implementation within this framework.
\end{abstract}


\tableofcontents

\section{Introduction}

The classical Enneper--Weierstrass representation expresses minimal surfaces in $\R^3$ in terms of holomorphic data, and has been extended in several directions to constant mean curvature (CMC) surfaces in other ambient geometries. In hyperbolic $3$-space $\HH$, Bryant's representation for $H=1$ \cite{Bryant1987} and Kenmotsu-type representation formulas provide foundational examples of such constructions.

In the range $0\le H<1$, several approaches have been developed, including the work of Aiyama and Akutagawa \cite{AiyamaAkutagawa2000} and loop group formulations \cite{DorfmeisterInoguchiKobayashi2015, DorfmeisterPeditWu1998}. These formulations reveal a rich integrable-systems structure underlying CMC surface theory in hyperbolic space.

In this paper, we present a representation framework for conformal CMC immersions with $0\le H<1$ based on three geometric ingredients:
\begin{itemize}
\item the Hermitian model of $\HH$;
\item a balanced spectral deformation;
\item Iwasawa splitting of $\SL$.
\end{itemize}

Our main result shows that such immersions arise locally from a rank-one $(1,0)$-form $\eta$ and a constant parameter $\lambda\in\C^*$ through a flat connection
\[
S^{-1}dS=\eta-\lambda\,\eta^*,
\]
with the mean curvature determined by $|\lambda|$. The flatness condition plays a central role: it is precisely the integrability condition which encodes the Gauss--Codazzi equations in this formulation.

We establish an explicit correspondence with the representation of Aiyama and Akutagawa via a gauge transformation, showing that the two formulations are equivalent at the level of flat connections. The construction also gives a transparent interpretation in terms of Kokubu's adjusted normal Gauss map and its associated metric structure.

To illustrate the framework, we present simple classes of flat rank-one data leading to local model solutions. These examples emphasize that the rank-one condition alone is not sufficient; the flatness condition imposes additional differential constraints on the admissible data.

Finally, we discuss aspects of monodromy, period problems, and numerical implementation, including a Magnus expansion scheme for integrating the flat connection and practical checks for flatness and unitarization.

\medskip
\noindent\textbf{Terminological note.}
We thank Professor Kazuo Akutagawa for helpful comments concerning the terminology and historical attribution of the Gauss map used in this setting. In the present paper, we use the terminology ``adjusted normal Gauss map'' in the sense of the representation framework developed by Aiyama and Akutagawa. We also acknowledge related earlier work of Kokubu on normal Gauss maps for minimal surfaces in hyperbolic space. These clarifications are terminological and do not affect the results of the paper.

\medskip
\noindent\textbf{Organization of the paper.}
\Cref{sec:geom-models} introduces the Hermitian model and the Lax pair formulation. \Cref{sec:spectral-iwasawa} develops the balanced spectral deformation and Iwasawa splitting. \Cref{sec:gauss-map} recalls the adjusted normal Gauss map and Kokubu's metric. \Cref{sec:main-theorem} proves the representation theorem. \Cref{sec:AA-dictionary} establishes the correspondence with Aiyama and Akutagawa. Later sections discuss global aspects, examples, stability considerations, and numerical implementation.

\section{Geometric setup and Lax pair}
\label{sec:geom-models}

\subsection{Hermitian model of $\HH$ and frames}

We identify $\R^{3,1}$ with $\Herm$ via
\[
(x^0,x^1,x^2,x^3)\longleftrightarrow
X=
\begin{pmatrix}
x^0+x^3 & x^1+ix^2\\
x^1-ix^2 & x^0-x^3
\end{pmatrix}\cdot
\]
The Lorentz product is
\[
\langle X,Y\rangle=-\frac12\,\tr(X\,\sigma_2\,Y^t\,\sigma_2),
\qquad
\sigma_2=
\begin{pmatrix}
0&i\\
-i&0
\end{pmatrix}\cdot
\]
Then
\[
\HH=\{X\in\Herm:\langle X,X\rangle=-1,\ X_{00}>0\},
\]
and $\SL$ acts on $\HH$ by
\[
g\cdot X=gXg^*,\qquad g\in\SL.
\]
Thus $\HH\simeq \SL/\SU$.

Let $f:M\to\HH$ be a conformal immersion with metric
\[
ds^2=e^{2u}|dz|^2.
\]
A moving frame $F:M\to\SL$ may be chosen so that
\begin{equation}\label{eq:geomframe}
e^{-u}f_z=
F\begin{pmatrix}0&1\\0&0\end{pmatrix}F^*,
\qquad
e^{-u}f_{\bar z}=
F\begin{pmatrix}0&0\\1&0\end{pmatrix}F^*,
\qquad
n=
F\begin{pmatrix}1&0\\0&-1\end{pmatrix}F^*.
\end{equation}

\subsection{Lax pair and Gauss--Codazzi}

Write
\[
F^{-1}dF=A\,dz+B\,d\bar z,
\]
where
\begin{equation}\label{eq:AB}
A=
\begin{pmatrix}
\frac12 u_z & \frac12 e^u(1+H)\\[2pt]
-\frac14 e^{-u}Q & -\frac12 u_z
\end{pmatrix},
\qquad
B=
\begin{pmatrix}
-\frac12 u_{\bar z} & \frac14 e^{-u}\bar Q\\[2pt]
\frac12 e^u(1-H) & \frac12 u_{\bar z}
\end{pmatrix}.
\end{equation}

\begin{lemma}[Gauss--Codazzi]\label{lem:GC}
The flatness equation
\[
A_{\bar z}-B_z-[A,B]=0
\]
is equivalent to
\begin{align}
u_{z\bar z}
-\frac{e^{2u}}{4}(1-H^2)
-\frac{e^{-2u}}{16}Q\bar Q&=0,
\label{eq:gauss}\\
Q_{\bar z}&=2e^{2u}H_z.
\label{eq:codazzi}
\end{align}
In particular, for a conformal immersion in $\HH$, the mean curvature $H$ is constant if and only if $Q$ is holomorphic.
\end{lemma}

\begin{proof}
The statement follows by substituting \eqref{eq:AB} into the flatness equation. The off-diagonal entries yield the Codazzi equation and its conjugate, while the diagonal part yields the Gauss equation. In particular, \eqref{eq:codazzi} gives $Q_{\bar z}=0$ precisely when $H$ is constant.
\end{proof}

\begin{remark}
For $H=0$ and normalized Hopf differential $Q\equiv2$, \eqref{eq:gauss} becomes
\[
4u_{z\bar z}-\cosh(2u)=0,
\]
the cosh-Gordon equation. This contrasts with the sinh-Gordon equation arising in the Euclidean CMC setting and reflects the distinction between the regimes $H>1$ and $0\le H<1$ in $\HH$.
\end{remark}

\section{Spectral deformations, balanced gauge, and Iwasawa splitting}
\label{sec:spectral-iwasawa}

\subsection{Balanced spectral deformation}

Following \cite{DorfmeisterInoguchiKobayashi2015,Toda2005}, consider the spectral changes
\[
(1\pm H)\mapsto s^{\pm1}(1\pm H),
\qquad s>0,
\]
together with the phase deformation
\[
Q\mapsto \theta^{-2}Q,
\qquad \theta\in S^1.
\]
Conjugating by
\[
g(\theta)=i
\begin{pmatrix}
0&\theta^{1/2}\\
\theta^{-1/2}&0
\end{pmatrix},
\qquad
R_{\pi/4}=
\begin{pmatrix}
e^{-i\pi/4}&0\\
0&e^{i\pi/4}
\end{pmatrix},
\]
collects the $S^1$ phase in the off-diagonal terms and introduces the factors needed for comparison with the unitary part of the Iwasawa decomposition.

\subsection{Iwasawa splitting and adjusted frame}

\begin{definition}[Iwasawa splitting]\label{def:iwasawa}
Let
\[
S=
\left\{
\begin{pmatrix}
a&w\\
0&a^{-1}
\end{pmatrix}
:\ a>0,\ w\in\C
\right\}
\subset \SL.
\]
Multiplication
\[
S\times \SU\longrightarrow \SL
\]
is a diffeomorphism. Thus every $F\in\SL$ decomposes uniquely as
\[
F=F_s\,\Phi,
\qquad
F_s\in S,\quad \Phi\in\SU.
\]
We call $\Phi$ the adjusted unitary frame.
\end{definition}

Since $\Phi\in\SU$, its Maurer--Cartan form satisfies
\[
\Phi^{-1}d\Phi\in su(2).
\]
Matching the off-diagonal magnitudes in the balanced Lax pair with the unitary part of the Iwasawa splitting gives the following condition.

\begin{proposition}[Balanced modulus]\label{prop:balance}
The balanced spectral parameter satisfies
\begin{equation}\label{eq:sH}
s=\sqrt{\frac{1-H}{1+H}}\in(0,1],
\qquad
H=\frac{1-s^2}{1+s^2}\cdot
\end{equation}
The phase parameter gives the associated $S^1$-family.
\end{proposition}

\begin{proof}
After balancing, the off-diagonal magnitudes are
\[
s^{\pm1}\frac{e^u}{2}(1\pm H).
\]
On the other hand, the adjusted unitary frame has off-diagonal magnitude
\[
\frac{e^u}{2}\sqrt{1-H^2}.
\]
Thus
\[
s(1+H)=\sqrt{1-H^2},
\qquad
s^{-1}(1-H)=\sqrt{1-H^2}.
\]
These equations are equivalent to \eqref{eq:sH}. The phase remains free and parameterizes the associated family.
\end{proof}

\section{Adjusted normal Gauss map and Kokubu's metric}
\label{sec:gauss-map}

\begin{definition}[Adjusted normal Gauss map]\label{def:kokubu}
Let $\Phi$ be the adjusted unitary frame from the Iwasawa splitting. We define
\[
\mathcal G=-\Phi\,\sigma_3\,\Phi^*\in S^2,
\qquad
\sigma_3=
\begin{pmatrix}
1&0\\
0&-1
\end{pmatrix}\cdot
\]
Its stereographic coordinate gives the adjusted normal Gauss map in the sense of Aiyama-Akutagawa, closely related to Kokubu's normal Gauss map.
\end{definition}

Aiyama and Akutagawa \cite{AiyamaAkutagawa2000} showed that the adjusted normal Gauss map is harmonic with respect to a metric of the form
\begin{equation}\label{eq:kokubu}
h_H=
\frac{4\,|d\zeta|^2}
{(1+|\zeta|^2)\big((1-|\zeta|^2)+H(1+|\zeta|^2)\big)}\cdot
\end{equation}
For $H=0$, this becomes the Kobayashi metric
\[
h_0=\frac{4|d\zeta|^2}{1-|\zeta|^4}\cdot
\]
For $0<H<1$, the metric has a singular circle
\[
|\zeta|=\sqrt{\frac{1+H}{1-H}} \cdot
\]

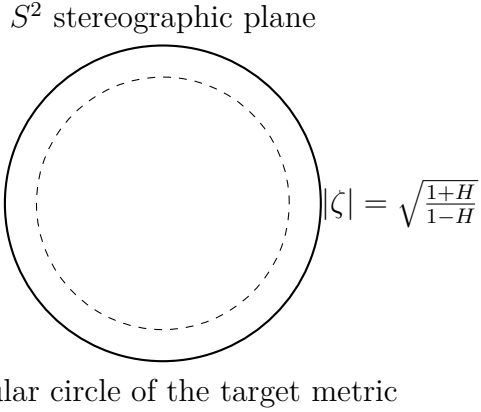
\begin{figure}[!t]
\centering
\begin{tikzpicture}[scale=1.05]
\draw[thick] (0,0) circle (2);
\draw[dashed] (0,0) circle (1.6);
\node at (0,2.35) {$S^2$ stereographic plane};
\node at (3.0,0) {$|\zeta|=\sqrt{\frac{1+H}{1-H}}$};
\node at (0,-2.4) {Singular circle of the target metric};
\end{tikzpicture}
\caption{The singular circle of Kokubu's target metric for $0<H<1$.}
\end{figure}

\section{Main representation theorem}
\label{sec:main-theorem}

\begin{definition}[Rank-one $(1,0)$-forms]\label{def:rankone}
A $(1,0)$-form
\[
\eta\in\mathfrak{sl}(2,\C)\otimes\Omega^{1,0}(M)
\]
is called rank-one if
\[
\det\eta\equiv0.
\]
Equivalently, locally
\[
\eta=\mathbf v\,\mathbf w^t\,dz,
\qquad
\mathbf w^t\mathbf v=0.
\]
\end{definition}

\begin{theorem}[Weierstrass--Kenmotsu representation for $0\le H<1$]\label{thm:main}
Let $M$ be simply connected. Suppose that $\eta$ is a rank-one $(1,0)$-form and that $\lambda\in\C^*$ is constant with $|\lambda|\le1$. Define
\begin{equation}\label{eq:S}
\Omega=\eta-\lambda\,\eta^*.
\end{equation}
Assume that $\Omega$ satisfies the flatness condition
\[
d\Omega+\Omega\wedge\Omega=0.
\]
Then there exists a map $S:M\to\SL$ such that
\[
S^{-1}dS=\Omega.
\]
Moreover,
\[
f=SS^*:M\to\HH
\]
is a conformal CMC immersion with
\[
H=\frac{1-|\lambda|^2}{1+|\lambda|^2}\in[0,1).
\]

Conversely, every conformal CMC immersion $f:M\to\HH$ with $0\le H<1$ arises locally from some rank-one $(1,0)$-form $\eta$ and constant $\lambda\in\C^*$ with
\[
|\lambda|=\sqrt{\frac{1-H}{1+H}},
\]
such that the associated connection $\Omega=\eta-\lambda\eta^*$ is flat. Varying $\arg\lambda$ yields the associated $S^1$-family.
\end{theorem}

\begin{proof}
Since $\Omega$ is assumed to be flat, the equation
\[
S^{-1}dS=\Omega
\]
integrates locally on the simply connected domain $M$.

Apply the Iwasawa decomposition
\[
S=F_s\,\Phi,
\qquad
F_s\in S,\quad \Phi\in\SU.
\]
The unitary factor $\Phi$ determines the adjusted normal Gauss map, and comparison of the off-diagonal terms with the balanced Lax pair gives
\[
|\lambda|=\sqrt{\frac{1-H}{1+H}}\cdot
\]
Equivalently,
\[
H=\frac{1-|\lambda|^2}{1+|\lambda|^2}.
\]

The flatness condition is the Maurer--Cartan equation for the integrated frame and therefore gives the Gauss--Codazzi equations. Hence the resulting map
\[
f=SS^*
\]
is a conformal CMC immersion in $\HH$.

Conversely, start with a conformal CMC immersion in $\HH$ and its moving frame. Its Maurer--Cartan form is flat. Applying the balanced spectral deformation and Iwasawa splitting described above, the $(1,0)$ part gives a rank-one form $\eta$, while the $(0,1)$ part is written as $-\lambda\eta^*$ with
\[
|\lambda|=\sqrt{\frac{1-H}{1+H}}\cdot
\]
Thus the connection $\Omega=\eta-\lambda\eta^*$ is flat by construction, and the immersion is recovered locally from $S^{-1}dS=\Omega$.
\end{proof}

\begin{corollary}[Minimal case]
The case $H=0$ corresponds to $|\lambda|=1$. Thus the associated family is obtained by varying $\arg\lambda$.
\end{corollary}

\section{Correspondence to the Aiyama-Akutagawa formula}
\label{sec:AA-dictionary}

Let $\nu:M\to S^2$ denote the adjusted normal Gauss map in stereographic coordinates. Following Aiyama and Akutagawa \cite{AiyamaAkutagawa2000}, define
\[
\omega=
\frac{-2\,\overline{\nu}_z}
{\sqrt{1-H^2}(1-|\nu|^4)}\,dz,
\qquad
\alpha=
\begin{pmatrix}
-\nu&\nu^2\\
-1&\nu
\end{pmatrix}\omega.
\]
Their connection form can be written as
\[
\tau=
\frac12\big((1+H)\alpha+(1-H)\alpha^*\big)
+\frac{\sqrt{1-H^2}}{4}[\sigma_3,\alpha+\alpha^*],
\]
and
\[
F^{-1}dF=\tau
\]
gives the immersion
\[
f=FF^*.
\]

\begin{proposition}[Gauge correspondence]\label{prop:dictionary}
The Aiyama--Akutagawa form $\tau$ is gauge-equivalent to a connection of the form
\[
\eta-\lambda\eta^*,
\]
where $\eta$ is rank-one and
\[
|\lambda|=\sqrt{\frac{1-H}{1+H}}\cdot
\]
Consequently, the two reconstructions agree after Iwasawa splitting.
\end{proposition}

\begin{proof}
The form $\tau$ decomposes into its $(1,0)$ and $(0,1)$ parts:
\[
\tau=\tau^{(1,0)}+\tau^{(0,1)},
\qquad
\tau^{(0,1)}=(\tau^{(1,0)})^*.
\]
After applying the balanced gauge, the nonunitary $(1,0)$ part is rank-one and can be denoted by $\eta$. The corresponding $(0,1)$ part is then $-\lambda\eta^*$, with $|\lambda|$ fixed by the balanced modulus condition. Since both connections arise from the same flat frame construction, their integrations and Iwasawa splittings yield the same immersion.
\end{proof}
\section{Global questions: periods, monodromy, and completeness}
\label{sec:global}

The representation theorem above is local on simply connected domains. On non-simply connected Riemann surfaces, one must impose the usual period and monodromy conditions.

\subsection{Monodromy}

Let $M$ be non-simply connected and let $\widetilde M$ be its universal cover. A flat connection
\[
\Omega=\eta-\lambda\eta^*
\]
integrates on $\widetilde M$ to a map
\[
S:\widetilde M\to\SL.
\]
For each $\gamma\in\pi_1(M)$, the monodromy is defined by
\[
S(\gamma\cdot p)=S(p)\rho(\gamma),
\qquad
\rho(\gamma)\in\SL.
\]
The immersion
\[
f=SS^*
\]
descends to $M$ if the monodromy satisfies
\[
\rho(\gamma)\rho(\gamma)^*=I
\]
for every $\gamma\in\pi_1(M)$, equivalently
\[
\rho(\gamma)\in\SU.
\]

Thus the global closing problem is the problem of adjusting the data so that the monodromy representation is unitarizable.

\subsection{Completeness and ends}

Completeness questions depend on the asymptotic behavior of the metric induced by the integrated frame. In the Aiyama--Akutagawa variables, the induced metric can be written in the form
\[
ds^2=(1+|\nu|^2)^2\,\omega\,\bar\omega,
\]
where $\nu$ is the adjusted normal Gauss map in stereographic coordinates and $\omega$ is the corresponding Weierstrass-type one-form.

Near punctures, completeness is therefore governed by the order of $\omega$, the behavior of $\nu$, and the approach of the adjusted normal Gauss map to the singular circle of the target metric. A full global theory requires simultaneous control of the flatness condition, the asymptotics of the metric, and the unitarization of monodromy.

\section{Local examples}
\label{sec:local-examples}

We illustrate the representation by a simple class of rank-one data for which the flatness condition can be checked explicitly. These examples are local models; their purpose is to show how the flatness condition restricts the rank-one data.

\begin{example}[A flat rank-one seed depending on one real variable]
Let $z=x+iy$ and let $\lambda>0$ be real. Define
\[
A(x)=\rho(x)
\begin{pmatrix}
-g(x)&g(x)^2\\
-1&g(x)
\end{pmatrix},
\qquad
\eta=A(x)\,dz.
\]
Then
\[
\det A(x)\equiv 0,
\]
so $\eta$ is a rank-one $(1,0)$-form. Consider the connection
\[
\Omega=\eta-\lambda\eta^*.
\]

A direct computation gives
\[
d\Omega+\Omega\wedge\Omega=0
\]
provided
\begin{align}
g'(x)&=-\frac{2\lambda}{1+\lambda}\,\rho(x)\,(1+g(x)^2)^2,
\label{eq:flat-g}\\
\rho'(x)&=\frac{4\lambda}{1+\lambda}\,g(x)\rho(x)^2(1+g(x)^2).
\label{eq:flat-rho}
\end{align}
Thus the flatness condition reduces in this example to an ordinary differential system.

A particularly simple explicit solution is obtained by taking
\[
\rho(x)=\frac{C}{1+g(x)^2},
\]
where $C$ is a nonzero real constant. Then \eqref{eq:flat-g} reduces to
\[
g'(x)=-\frac{2\lambda C}{1+\lambda}(1+g(x)^2),
\]
and hence
\[
g(x)=
\tan\!\left(
-\frac{2\lambda C}{1+\lambda}x+\delta
\right)
\]
on any interval avoiding the poles of the tangent function.

With this choice, $\Omega$ is flat. Therefore, by \Cref{thm:main}, the equation
\[
S^{-1}dS=\Omega
\]
integrates locally and yields a conformal CMC immersion
\[
f=SS^*:M\to\HH.
\]
\end{example}

\begin{remark}
This example also shows why the flatness assumption is essential. The rank-one condition alone does not imply flatness; the null direction and the scalar factor must satisfy compatibility equations such as \eqref{eq:flat-g}--\eqref{eq:flat-rho}.
\end{remark}

\section{A cylindrical model}
\label{sec:cylindrical}

The previous construction can be placed on a cylinder. Let
\[
w=x+iy,
\qquad
y\sim y+2\pi,
\]
so that
\[
M\simeq \R\times S^1.
\]
Let $\lambda>0$ be real and let $A(x)$ be as in \Cref{sec:local-examples}, with $g$ and $\rho$ satisfying \eqref{eq:flat-g}--\eqref{eq:flat-rho}. Define
\[
\eta=A(x)\,dw,
\qquad
\Omega=\eta-\lambda\eta^*.
\]
Then $\eta$ is rank-one and $\Omega$ is flat. Hence
\[
S^{-1}dS=\Omega
\]
integrates locally on the cylinder, and
\[
f=SS^*
\]
gives a local CMC immersion into $\HH$.

This example provides a simple cylindrical model in which the flatness condition reduces to an ordinary differential system. Global descent to the quotient cylinder depends, as usual, on the monodromy condition discussed in \Cref{sec:global}.

\section{Remarks on flatness and rank-one data}
\label{sec:flatness-remark}

The preceding examples illustrate that nontrivial flat rank-one data arise when the null direction varies in a compatible way. This is an important point: the condition
\[
\det\eta\equiv 0
\]
does not by itself imply
\[
d\Omega+\Omega\wedge\Omega=0.
\]

Indeed, if one takes a fixed nilpotent direction, for example
\[
\eta=a(z)E_{21}\,dz,
\]
then
\[
\Omega=a(z)E_{21}\,dz-\lambda\overline{a(z)}\,E_{12}\,d\bar z.
\]
Even when $a$ is holomorphic, so that the exterior derivative contributes no curvature away from singularities, the mixed term gives
\[
\Omega\wedge\Omega
=
\lambda |a(z)|^2 [E_{12},E_{21}]\,dz\wedge d\bar z.
\]
Since
\[
[E_{12},E_{21}]\ne 0,
\]
such a connection is not flat unless the seed is trivial or one is in a degenerate parameter case. Thus the flatness condition imposes genuine nonlinear constraints on the data.

This observation is consistent with the usual role of the Maurer--Cartan equation in surface theory: flatness is the integrability condition corresponding to the Gauss--Codazzi equations.

\section{Stability and second variation}
\label{sec:stability}

We recall the standard Jacobi operator for normal variations preserving constant mean curvature in $\HH$. Let
\[
V=\varphi n
\]
be a normal variation. The second variation of area with Lagrange multiplier enforcing the CMC condition is
\[
\delta^2\mathcal A_H(\varphi)
=
\int_M
\left(
|\nabla\varphi|^2-(|A|^2+\mathrm{Ric}(n,n))\varphi^2
\right)\,dA.
\]
In $\HH$, one has
\[
\mathrm{Ric}(n,n)=-2.
\]
Writing the extrinsic curvature contribution in terms of the Hopf differential gives
\[
|A|^2=2H^2+2K_e,
\qquad
K_e=-\frac14 e^{-4u}|Q|^2.
\]
Thus the Jacobi operator takes the form
\begin{equation}\label{eq:Jacobi}
\mathcal J\varphi
=
-\Delta\varphi
-
\left(
2H^2-\frac12 e^{-4u}|Q|^2-2
\right)\varphi.
\end{equation}

This formula is useful for studying local stability and spectral questions once explicit or numerical CMC data have been constructed from a flat connection.

\subsection{Fourier mode decomposition}

On a rotationally symmetric annular domain, one may write
\[
\varphi(r,t)=\sum_{m\in\mathbb Z}\phi_m(r)e^{imt}.
\]
In conformal coordinates, the Laplacian is
\[
\Delta=4e^{-2u}\partial_{z\bar z},
\]
and the Jacobi operator decomposes into Fourier modes. Setting
\[
s=\log r,
\]
one obtains a one-dimensional family of operators of the form
\[
\mathcal J_m\phi
=
-\frac{d^2\phi}{ds^2}-V_m(s)\phi,
\]
where
\[
V_m(s)
=
e^{2u(s)}
\left(
2H^2-\frac12e^{-4u(s)}|Q(s)|^2-2
\right)
-m^2.
\]
This reduction is standard and provides a useful way to study index and stability in symmetric examples.

\section{Numerical implementation with error control}
\label{sec:numerical}

We briefly indicate a numerical approach for integrating flat $\SL$-connections of the form
\[
\Omega=\eta-\lambda\eta^*.
\]
The purpose is to provide a practical framework for constructing local models and checking consistency of the data.

\begin{enumerate}[label=\textbf{(\arabic*)}, leftmargin=3em]

\item \textbf{Magnus integration.}
Use a fourth-order two-point Magnus expansion
\[
\Omega_{\mathrm{Magnus}}
=
\frac12(\Omega_1+\Omega_2)
-
\frac{\sqrt3}{12}[\Omega_1,\Omega_2],
\]
with Gauss--Legendre nodes. For any traceless $2\times2$ matrix $X$,
\[
\exp X
=
\cosh\sigma\,I+\frac{\sinh\sigma}{\sigma}X,
\qquad
\sigma=\sqrt{\frac12\tr(X^2)}.
\]
After each step, one may renormalize to maintain $\det S=1$.

\item \textbf{Adaptive step control.}
An embedded local error estimator can be used to adapt the step size. If $\epsilon$ denotes the normalized local error, a standard update is
\[
h_{\mathrm{new}}
=
\min\{2,\max\{1/2,0.9\,\epsilon^{-1/4}\}\}\,h.
\]

\item \textbf{Iwasawa or polar decomposition.}
At each step, compute a polar or Iwasawa-type factorization of $S$ to obtain the unitary factor $\Phi$. The immersion and adjusted normal Gauss map are then recovered from
\[
f=SS^*,
\qquad
\mathcal G=-\Phi\sigma_3\Phi^*.
\]

\item \textbf{Flatness residuals.}
For numerical data, one should monitor the residual
\[
d\Omega+\Omega\wedge\Omega
\]
as a consistency check. This is essential, since the rank-one condition alone does not guarantee flatness.

\item \textbf{Monodromy.}
On non-simply connected domains, compute the holonomy matrices around generators of $\pi_1(M)$. The descent condition requires these holonomies to be unitary, or at least unitarizable by an admissible gauge.

\end{enumerate}

\section{Applications and further directions}
\label{sec:applications}

Several directions arise naturally from this framework.

\begin{enumerate}[label=\textbf{(\arabic*)}, leftmargin=3em]

\item \textbf{Willmore viewpoint.}
As $H\to0$, the target metric for the adjusted normal Gauss map approaches the Kobayashi metric. This connects the present construction to the Willmore surface viewpoint and to harmonic map methods.

\item \textbf{Flatness-constrained moduli.}
The examples above show that rank-one data alone are not sufficient. The flatness condition imposes additional differential constraints on the scalar factor and the null direction defining $\eta$. Understanding the moduli of admissible flat rank-one data remains an interesting problem.

\item \textbf{Global closing problems.}
On non-simply connected domains, one must solve a monodromy unitarization problem. This is the natural analog of period closing in classical Weierstrass representations.

\item \textbf{Real forms and related geometries.}
The balancing principle is closely related to other CMC surface theories, including the cases $H>1$ in hyperbolic space, CMC surfaces in $S^3$, and timelike CMC surfaces in Lorentzian space forms.

\item \textbf{Numerical continuation.}
The explicit flatness equations and the Magnus integration scheme make it possible to explore families of local solutions numerically, while monitoring both determinant preservation and flatness residuals.

\end{enumerate}

\section{Concluding remarks}

We presented a Weierstrass--Kenmotsu type representation for conformal CMC immersions with $0\le H<1$ in $\HH$. The construction uses rank-one $(1,0)$ data together with a constant spectral parameter $\lambda$, with mean curvature encoded by
\[
H=\frac{1-|\lambda|^2}{1+|\lambda|^2}\cdot
\]
The essential integrability requirement is the flatness of the associated connection
\[
\Omega=\eta-\lambda\eta^*.
\]
This condition is the geometric content of the Gauss--Codazzi equations in the representation.

The Iwasawa splitting clarifies the role of the adjusted normal Gauss map and connects the formulation to the Aiyama--Akutagawa representation. Local model examples show how the flatness condition leads to explicit differential constraints on the rank-one data. Global questions, including monodromy and completeness, remain naturally tied to the same flatness and unitarization conditions.

\bigskip

\noindent\textbf{Acknowledgements.}
The authors thank Professor Kazuo Akutagawa for helpful comments concerning terminology and historical attribution related to the adjusted normal Gauss map. They also thank Professors Josef Dorfmeister, Franz Pedit, Alexander Bobenko, Wayne Rossman, and Junichi Inoguchi for valuable discussions over many years.

\appendix

\section{Balancing derivation and phase bookkeeping}
\label{app:balancing}

We record the derivation leading to the balanced modulus relation. Starting from the Lax pair in \eqref{eq:AB}, apply the transformations
\[
(1\pm H)\mapsto s^{\pm1}(1\pm H),
\qquad
Q\mapsto\theta^{-2}Q.
\]
Conjugation by
\[
g(\theta)=i
\begin{pmatrix}
0&\theta^{1/2}\\
\theta^{-1/2}&0
\end{pmatrix}
\]
moves the phase into the off-diagonal entries. The additional diagonal conjugation
\[
R_{\pi/4}
=
\begin{pmatrix}
e^{-i\pi/4}&0\\
0&e^{i\pi/4}
\end{pmatrix}
\]
introduces the factors needed for comparison with the unitary frame. The resulting off-diagonal magnitudes are
\[
s^{\pm1}\frac{e^u}{2}(1\pm H).
\]
Equating these with
\[
\frac{e^u}{2}\sqrt{1-H^2}
\]
gives
\[
s=\sqrt{\frac{1-H}{1+H}},
\qquad
H=\frac{1-s^2}{1+s^2}\cdot
\]

\section{Symbolic exponentials}
\label{app:exp}

For a traceless $2\times2$ matrix $X$, one has
\[
X^2=\frac12\tr(X^2)\,I.
\]
Thus, setting
\[
\sigma=\sqrt{\frac12\tr(X^2)},
\]
we obtain
\[
\exp X
=
\cosh\sigma\,I+\frac{\sinh\sigma}{\sigma}X.
\]
In particular, for
\[
X=\alpha E_{21}+\beta E_{12},
\]
one has
\[
X^2=(\alpha\beta)I,
\qquad
\sigma=\sqrt{\alpha\beta}.
\]

\section{Jacobi operator in Aiyama--Akutagawa variables}
\label{app:AAjacobi}

We express the Jacobi operator \eqref{eq:Jacobi} in terms of the Aiyama--Akutagawa variables. Recall that
\[
\omega=
\frac{-2\,\overline{\nu}_z}
{\sqrt{1-H^2}(1-|\nu|^4)}\,dz,
\qquad
\alpha=
\begin{pmatrix}
-\nu&\nu^2\\
-1&\nu
\end{pmatrix}\omega.
\]
The induced metric has the form
\[
ds^2=(1+|\nu|^2)^2\,\omega\,\bar\omega.
\]
Writing
\[
ds^2=e^{2u}|dz|^2,
\]
we obtain
\[
e^{2u}=(1+|\nu|^2)^2|\omega|^2.
\]

The Hopf differential can be read from the $(2,1)$-entry of the $(1,0)$ part of the connection. In the present normalization this gives
\[
Q=C_H e^u\omega,
\]
where
\[
C_H=2(1+H)-\sqrt{1-H^2}.
\]
Consequently,
\[
|Q|^2e^{-4u}
=
|C_H|^2e^{-2u}|\omega|^2
=
\frac{|C_H|^2}{(1+|\nu|^2)^2}.
\]
Substituting into \eqref{eq:Jacobi} yields
\[
\mathcal J\varphi
=
-\Delta\varphi
-
\left(
2H^2-2
-\frac12\frac{|C_H|^2}{(1+|\nu|^2)^2}
\right)\varphi.
\]

\begin{remark}
The coefficient $C_H$ depends on the chosen balanced normalization. Other consistent normalizations lead to equivalent expressions after rescaling the corresponding Weierstrass data.
\end{remark}


\bigskip

\noindent
\textbf{Magdalena Toda}\\
Department of Mathematics and Statistics, Texas Tech University,\\
Lubbock, TX 79409, USA\\
\textit{Email:} \texttt{magda.toda@ttu.edu}

\bigskip

\noindent
\textbf{Erhan Güler}\\
Department of Mathematics and Statistics, Texas Tech University,\\
Lubbock, TX 79409, USA\\
\textit{Email:} \texttt{eguler@ttu.edu}

\bigskip

\noindent
\textbf{Madusha Dilhani Atampalage}\\
Department of Computer Systems Engineering, University of Kelaniya,\\
Kelaniya 11600, Sri Lanka\\
\textit{Email:} \texttt{madusha.chandrasena@gmail.com}

\end{document}